\documentclass[preprint,3p]{elsarticle}  
\usepackage{xcolor}
\usepackage[T1]{fontenc}
\usepackage[utf8]{inputenc}
\usepackage{amsmath,amsfonts,amsthm,bm}
\DeclareMathOperator{\sgn}{sgn}
\usepackage{amssymb}
\usepackage{caption} 
\usepackage{bigints}
 \usepackage[english]{babel} 
 \usepackage{float} 
 \usepackage{graphics} 
 \usepackage{graphicx} 
\usepackage[utf8]{inputenc} 
\usepackage{mathtools}
\usepackage{capt-of} 
\usepackage{calrsfs}
\usepackage{lipsum}  
 \usepackage{epsfig}
\newcommand{\R}{\mathbf{R}}

\newtheorem{theorem}{Theorem}[section]
\newtheorem{exam}{Example}[section]
\newtheorem{lemma}{Lemma}[section]
\newtheorem{cor}{Corollary}[section]
\newtheorem{definition}{Definition}[section]
\newtheorem{prop}{Proposition}[section]

\newdefinition{remark}{Remark}[section]

\DeclareMathAlphabet{\pazocal}{OMS}{zplm}{m}{n}

\numberwithin{equation}{section}
\begin{document}%

\begin{frontmatter}
\title{On the basic theory of some generalized and fractional derivatives} 
\author{Leila Gholizadeh Zivlaei\footnote[1]{leilagh@math.carleton.ca} and Angelo B. Mingarelli\footnote[2]{angelo@math.carleton.ca}}
\address{School of Mathematics and Statistics, Carleton University, Ottawa, Canada}

\begin{abstract}
We continue the development of the basic theory of generalized derivatives as introduced in \cite{JPA} and give some of their applications. In particular, we formulate versions of a weak maximum principle, Rolle's theorem, the Mean value theorem, the Fundamental theorem of Calculus, Integration by parts, along with an existence and uniqueness theorem for a generalized Riccati equation, each of which includes, as corollaries, the corresponding version for conformable fractional derivatives considered by  \cite{kat}, \cite{kha} among many others. Finally, we show that for each $\alpha > 1$ there is a fractional derivative and a corresponding function whose fractional derivative fails to exist everywhere on the real line.  
\end{abstract}

\end{frontmatter}

\section{Introduction}
In a previous paper \cite{JPA} generalized (fractional) derivatives of the form 
\begin{equation}\label{eqq01}
D_p f(t) = \lim_{h\to 0} \frac{f(p(t,h))-f(t)}{h}
\end{equation}
were introduced whenever the limit exists and is finite. Here $p$ is a real valued function of the two real variables $(t, h)$ where $t$ is defined on some interval $I\subseteq \R$. Recently, two special cases of fractional derivatives of order $\alpha$ (with $0 < \alpha <1$) were defined by means of the limit definition,
\begin{equation}\label{eqq02}
D^{\alpha} f(t) = \lim_{h\to 0} \frac{f(p(t,h,\alpha))-f(t)}{h}
\end{equation}
 where $p(t,h,\alpha) = t+ht^{1-\alpha}$ see Katugampola \cite{kat}, and where $p(t,h,\alpha) = t\exp(ht^{-\alpha})$, see Khalil \cite{kha}. Under mild conditions on $p$ it can be shown (see \cite{JPA}) that such $p$-derivatives satisfy all the basic rules of differentiation (Product rule, etc.) and that indeed, the derivatives defined in \cite{kat} and \cite{kha} are special cases of these where, however, $p_h(t,0)\neq 0$ ($p_h$ is the partial derivative of $p$ with respect to its second variable).
 
We know (see  \cite{JPA}) that whenever $f$ is a differentiable function (in the ordinary sense, i.e., when $p(t,h) = t+h$) and $p_h(t,0) \neq 0$, then its generalized derivative \eqref{eqq01} exists at $t$ and 
\begin{equation}\label{eqq03}
D_p f(t) = p_h(t,0) \, f^\prime(t). 
\end{equation}
where the prime is the ordinary derivative of $f$ at $t$. This shows that the existence of a derivative implies the existence of a $p$-derivative so long as $p_h(t,0) \neq 0$. Indeed, \eqref{eqq03} shows that generalized or even fractional derivatives defined by either \eqref{eqq01} and \eqref{eqq02} are simply multiplication operators on the space of derivatives of differentiable functions, or just {\it weighted} first order differential operators where the weight, represented by $p_h(t,0)$, may itself be sign-indefinite on $I$.

Physically, one may think of these generalized derivatives as momentum operators with the $p$-term serving as a mass of arbitrary sign. Applications of the foregoing theory can be found in classical mechanics, see \cite{JPA}, where it is shown, in part, that in celestial mechanics elliptical orbits prevail independently of the fractional derivative chosen. The  resulting theory can then be included within the framework of weighted operator theory of ordinary differential operators, not without its own difficulties especially when the indicated partial derivative changes sign on intervals or, more generally, on sets of positive measure.

By way of examples, there has been a flurry of activity of late in the area of locally defined fractional differential operators and corresponding equations. Among these we cite chronologically, \cite{kha, kat, as}. We cannot begin to cite all the references related to these as, for example, Google Scholar refers to more than two-thousand references to paper \cite{kha} alone! Still, the articles referred to here will suffice for our purpose. In the sequel, unless otherwise specified, $\alpha, \beta$ will denote real parameters with $0 < \alpha<1$, $\beta \neq 0$, and $\beta \not\in \mathbf{Z}^-$.

Khalil, {\it et al.} \cite{kha}  define a function $f$ to be $\bm{\alpha}$${\bf -differentiable}$ at $t>0$ if the limit, 
\begin{equation}
\label{eq001}
T_0^\alpha f(t) = \lim_{ h\to 0} \frac{f(t+h t^{1-\alpha})-f(t)}{h}.
\end{equation}
exists and is finite. Katugampola \cite{kat} presented another (locally defined) derivative by requiring that, for $t>0$, 
\begin{equation}
\label{eq002}
T_0^\alpha f(t) = \lim_{ h\to 0} \frac{f(te^{h t^{-\alpha}})-f(t)}{h}.
\end{equation}
exist and be finite.  Recently, other authors, e.g., \cite{as}, considered minor variations in the definition \eqref{eq001} by assuming that, for $t>0$, 
\begin{equation}
\label{eq003}
D^{GFD} f(t) = \lim_{ h\to 0} \frac{f(t+\frac{\Gamma(\beta)\,  h t^{1-\alpha}}{\Gamma(\beta-\alpha+1)})-f(t)}{h}.
\end{equation}
exist and be finite. Note that in each of the three definitions, the case $\alpha=1$ (and $\beta=1$ in \eqref{eq003}) leads to the usual definition of a derivative (see Section~\ref{s2}  in the case of \eqref{eq002}).

Observe that, in each case, \eqref{eq001}, \eqref{eq002}, \eqref{eq003}, the derivatives are defined for $t>0$. As many authors have noticed one merely needs to replace $t^{1-\alpha}$ by  $(t-a)^{1-\alpha}$ in each of these definitions to allow for a derivative to be defined on an a given interval $(a, b)$, whether finite or infinite. In the case of \eqref{eq002} the point $t=0$ must be excluded as it is not in the domain of definition of the exponential.

More recently, G\'uzman {\it et al.}, \cite{guz}, considered a fractional derivative called the {\it N-derivative} defined for $t>0$ by
\begin{equation}
\label{eq002x}
N_F^\alpha f(t) = \lim_{ h\to 0} \frac{f(t  + h e^{ t^{-\alpha}})-f(t)}{h}.
\end{equation}
The function $f$ is said to be $N$-differentiable if \eqref{eq002x} exists and is finite. In a later paper \cite{viv}, the authors stated slight extensions of the previous definition to another $N_F^{\alpha}(t)$ derivative defined by the existence and finiteness of the limit,
\begin{equation}
\label{eq002y}
N_F^\alpha f(t) = \lim_{ h\to 0} \frac{f(t  + hF(t, \alpha))-f(t)}{h}.
\end{equation}
Here $F$ seems to have been left generally unspecified, yet the authors obtain the relation (see \cite{viv}, Equation (2))
\begin{equation}\label{eq002z} N_F^{\alpha}f(t) = F(t,\alpha)\, f^\prime(t),\end{equation}
also a consequence of Theorem~\ref{th2.5} below and the results in \cite{JPA} but whose validity is sensitive to the behaviour of the function $F$.

By way of another example, we show that the definition of the generalized derivative,
\begin{equation}
\label{eq009}
D^\alpha f(t) = \lim_{ h\to 0} \,  \frac{f(t+\sin(h) \, (\cos t)^{1-\alpha}) - f(t)}{h},
\end{equation}
where $0 < \alpha <1$ and the limit exists and is finite for $t\in [0,b]$, $b<\pi/2$, can also be handled by our methods and will lead to results that are a consequence of the definition in \cite{JPA}. Observe, in passing, that the case where $\alpha=1$ in \eqref{eq009}, i.e., $D^\alpha f(t) =f^\prime(t)$ whenever $f$ is differentiable in the ordinary sense.

These definitions of a generalized derivative function (sometimes called a fractional derivative) on the real line will be referred to occasionally as {\it locally defined} derivatives as, in each case, knowledge of $f$ is required merely in a neighborhood of the point $t$ under consideration. In contrast, in the case of the more traditional Riemann-Liouville or Caputo fractional derivatives, knowledge of the function $f$ is required on a much larger interval including the point $t$, see e.g., \cite{baj}. The previous, though very popular, derivatives defined in terms of singular integral operators shall not be considered here.

In this paper we continue the study of the generalized derivatives introduced in \cite{JPA} and show that a large number of  definitions such as \eqref{eq001}, \eqref{eq002}, \eqref{eq003}, \eqref{eq002x}, \eqref{eq002y} and \eqref{eq009} can be included in the more general framework of \cite{JPA}. Thus, the results obtained in papers using either of the fabove definitions are actually a corollary of our results. We also find versions of Rolle's Theorem, the Mean Value Theorem, and the Fundamental Theorem of Calculus, a result on the nowhere fractional differentiability of a class of fractional derivatives and corresponding functions, as well as an existence and uniqueness theorem for a generalized Riccati differential equation, all of which generalize many results in the literature.


\section{Preliminary observations}\label{s2}

The authors in each of \cite{kha, kat, as} point out that the notion of differentiability represented by either one of the definitions \eqref{eq001}, \eqref{eq002}, \eqref{eq003} considered, is more general than the usual one by showing that there are examples whereby $\alpha-$differentiability does not imply differentiability in the usual sense, although the converse holds. However, note that for $t>0$ the change of variable $s= t+h t^{1-\alpha}$ shows that
$$ \lim_{ h\to 0} \frac{f(t+h t^{1-\alpha})-f(t)}{h} = \frac{1}{t^{\alpha-1}}\lim_{s\to t} \frac{f(s)-f(t)}{s-t}.$$
Consequently, the left hand limit exists if and only if the right hand limit exists. Thus, for $t>0$, $p$-differentiability is equivalent to ordinary differentiability and the only possible exception is at $t=0$ (which is excluded anyway, by definition). A similar argument applies in the case of \eqref{eq003}.

Insofar as \eqref{eq002} is concerned, observe that solving for $h$ after the change of variable $s=te^{h t^{-\alpha}}$ is performed, and $t>0$, leads one to
$$\lim_{ h\to 0} \frac{f(te^{h t^{-\alpha}})-f(t)}{h} = \frac{1}{t^{\alpha}}\left( \lim_{s\to t} \frac{f(s)-f(t)}{s-t}\right ) \left( \lim_{s\to t} \frac{s-t}{\ln s-\ln t}\right ) = \frac{1}{t^{\alpha-1}}\lim_{s\to t} \frac{f(s)-f(t)}{s-t}.$$ 
It follows that both definitions, \eqref{eq001} and \eqref{eq002}, coincide for $t>0$ (see also \cite{JPA}). Definition \eqref{eq003} is simply a re-scaling of \eqref{eq001} by a constant  as can be verified by replacing $h$ in \eqref{eq001} by $\eta = h\, c$ where $c = \Gamma(\beta) / \Gamma(\beta-\alpha+1).$ Thus, strictly speaking, although it appears to be more general than \eqref{eq001}, it isn't really so. 

Next, observe that, for $t>0$, the $N$-derivative defined by \eqref{eq002x}, 
$$\lim_{ h\to 0} \frac{f(t + he^{t^{-\alpha}})-f(t)}{h} = \frac{1}{e^{-t^{\alpha}}}\left( \lim_{s\to t} \frac{f(s)-f(t)}{s-t}\right ) = e^{t^\alpha}\, f^\prime(t)$$ 
if either limit exists. Thus, for $t>0$, $f$ is $N$-differentiable if and only if $f$ is differentiable in the ordinary sense. This means that in [\cite{guz}, Theorem 2.3 (f)], $N$-differentiability is equivalent to ordinary differentiability (see the Remark in \cite{guz}, p.91).

Finally, for definition \eqref{eq009}, note that, for $t\in [0,b]$, $b < \pi/2$,
$$\lim_{ h\to 0} \,  \frac{f(t+\sin(h) \, \cos(t)^{1-\alpha}) - f(t)}{h} =\lim_{ s\to t} \,  \frac{f(s)-f(t)}{s-t}\lim_{ s\to t} \,  \frac{s-t}{\arcsin((s-t)\sec^{1-\alpha}(t)}=\cos^{1-\alpha}(t)\, f^\prime(t),$$
if either limit on the left or right exists.

Definition \eqref{eq009} gives something new that hasn't been studied before but falls within the framework of the theory developed in \cite{JPA}. These observations lead to the following theorem (see also Theorem 2.4 in \cite{JPA}).

Finally, for the generalized derivative defined by \eqref{eq002y}, it can be shown that this and the previous results are consequences of Proposition~\ref{th2.5} below.

\begin{theorem}\label{th1}
For $t>0$ (resp. $t\geq 0$) a function is {\it differentiable} in the sense of either \eqref{eq001}, \eqref{eq002}, or \eqref{eq003}, \eqref{eq002y}, \eqref{eq002z} (resp. \eqref{eq009}) if and only if it is differentiable in the usual sense.
\end{theorem}
\begin{remark}
It follows that the {\it conformable fractional derivatives} defined above (except for \eqref{eq009}) are actually differentiable in the ordinary sense whenever $t>0$, the only possible point of non-differentiability being at $t=0$.
\end{remark}

\section{Generalized derivatives}\label{s2a}
Let $ f : I \rightarrow\mathbb{R}$ , $ I \subseteq\mathbb{R}$ and $p : U_{\delta} \rightarrow\mathbb{R}$ where
$U_{\delta}= \{(t, h) : t \in I, |h| <\delta\}$ for some $\delta > 0$ is some generally unspecified neighborhood of $(t, 0)$. Unless $\delta$ is needed in a calculation we shall simply assume that this condition is always met. Of course, we always assume that the range of $p$ is contained in $I$. In the sequel, $L(I) \equiv  L_{1}(I)$ is the usual space of Lebesgue integrable functions on $I$. 

For a given $\alpha$, the {\bf generalized derivative} or $\mathbf{p}$-{\bf derivative} in \cite{JPA} is defined by 
\begin{equation}\label{eq005}
D_p f(t) = \lim_{h\to 0} \frac{f(p(t,h))-f(t)}{h}.
\end{equation}
whenever the limit exists and is finite. Occasionally, we'll introduce the parameter $\alpha$ mentioned above into the definition so that the limit
\begin{equation}\label{eq01}
D_p^{\alpha} f(t) = \lim_{h\to 0} \frac{f(p(t,h,\alpha))-f(t)}{h}.
\end{equation}
will then be called the $\mathbf{p}${\bf-derivative of order} ${\bm \alpha}$. Since $\alpha$ is a parameter \eqref{eq01} is actually a special case of \eqref{eq005}. 

The main hypotheses on the function $p$ are labeled $H1^{\pm}$ and $H2$ in \cite{JPA} and can be summarized together as follows:  

{\bf Hypothesis (H)}.  Given an interval $I\subset \R$, in addition to requiring that for $t\in I$, the function $p(t,h)$, $p_h(t,h)$ are continuous in a neighborhood of $h =0$, we ask that for $ t \in I$ and for all sufficiently small $\varepsilon > 0$, $p(t, h) = t \pm \varepsilon$ has a solution $h = h(t, \varepsilon)$ such that $ h\rightarrow 0$ as $\varepsilon\rightarrow 0$. 

Occasionally, we will assume, in addition, that $\frac{1}{p_{h}(\cdot , 0)}\in L(I).$ 

{\bf NOTATION:} The notations $D, D^\alpha, D_p, D_p^\alpha$ will be used interchangeably, occasionally for emphasis.



It is easy to see that the notion of $p$-differentiability as defined in \eqref{eq005} is very general in that it includes all the above definitions. To this end, it suffices to show that hypothesis (H) is satisfied and this is a simple matter, see also \cite{JPA}. 
\begin{theorem}\label{th2}
Each of the derivatives defined by \eqref{eq001}, \eqref{eq002}, \eqref{eq003}, \eqref{eq002x}, \eqref{eq002y}, and \eqref{eq009} above are $p$-derivatives for an appropriate function $p$ satisfying (H).
\end{theorem}

The following basic property is expected of a generalized derivative and indeed holds for the class considered here and in \cite{JPA}. 
\begin{theorem}\label{th1x} {\rm (See Theorem 2.1 in \cite{JPA}.)}
Let $p$ satisfy (H). If $f$ is $p$-differentiable at $a$ then $f$ is continuous at $a$.
\end{theorem}
\begin{cor} \label{cor01} Let $f$ be $\alpha$-differentiable where the $\alpha$-derivative is defined in either \eqref{eq001}, \eqref{eq002}, \eqref{eq003}, \eqref{eq002x}, \eqref{eq002y} or \eqref{eq009}. Then $f$ is continuous there.
\end{cor}
\begin{remark}We note that Corollary~\ref{cor01} includes Theorem 2.1 in \cite{kha}; Theorem 2.2 in \cite{kat}, and Theorem 2.2 in \cite{guz}).\end{remark}
The usual rules for differentiation are also valid in this more general scenario. In the sequel for a given function $p$ satisfying (H) we will write $D_p :=D$  or, if there is a parameter dependence, $D_p^\alpha :=D$, for simplicity. 
\begin{prop} \label{th2.2} {\rm (See Theorem 2.2 in \cite{JPA} for proofs.)}
\begin{itemize} 
\item[(a)] {\rm (The Sum Rule)}  If $f, g$ are both $p$-differentiable at $t\in I$ then so is their sum, $f + g$, and
$$D (f+ g)(t) =  D f(t) + D g(t).$$
\item[(b)] {\rm (The Product Rule)}  Assume that $p$ satisfies (H) and that for $t\in I$, $p(t, h)$ is continuous at $h=0$. If $f, g$ are both $p$-differentiable at $t\in I$ then so is their product, $f\cdot g$, and
$$D (f\cdot g)(t) =  f(t)\cdot D g(t) + g(t)\cdot D f(t). $$
\item[(c)] {\rm (The Quotient Rule)}  Assume that $p$ satisfies (H). If $f, g$ are both $p$-differentiable at $t\in I$ and $ g(t) \neq 0$ then so is their quotient, $f/g$, and
$$D \left (\frac{f}{g}\right)(t) =  \frac{g(t)\cdot D f(t) - f(t)\cdot D g(t)}{g(t)^2}. $$
\end{itemize}
\end{prop}
As a result we obtain,
\begin{cor}\label{th3} {\rm (See Theorem 2.2 in \cite{kha}; Theorem 2.3 in \cite{kat}; Theorem 4 in \cite{as}, and Theorem 2.3 in \cite{guz}.)}
For each of the definitions  \eqref{eq001}, \eqref{eq002}, \eqref{eq003}, \eqref{eq002x}, \eqref{eq002y} and \eqref{eq009} there holds an analog of the sum/product/and quotient rule for differentiation of corresponding $p$-derivatives.
\end{cor}

\begin{prop}\label{th2.4} {\rm (See \cite{JPA}, Theorem 2.4)}
Assume (H). Let $f$ be continuous and non-constant on $I$, and let $f$ be $p$-differentiable at $t\in I$. Let $g$ be defined on the range of $f$ and let $g$ be differentiable at $f(t)$. Then the composition $g\circ f $ is $p$-differentiable at $t$ and
\begin{equation*}
D (g\circ f )(t) = g^{\prime}(f(t))\,D f(t).
\end{equation*}
\end{prop}
\begin{cor} {\rm (See \cite{kha}, p.66 (iv), although the Chain Rule is not stated correctly there, and Theorem 2.3 in \cite{kat}.)}
For each of the definitions  \eqref{eq001}, \eqref{eq002}, \eqref{eq003}, \eqref{eq002x}, \eqref{eq002y} and \eqref{eq009} there holds an analog of the Chain rule for differentiation of corresponding $p$-derivatives in the form
$$ D  (g\circ f)(t) = g^\prime(f(t))\, D f(t),$$
where $D:=D^\alpha$ (resp. $D:=D_p^{\alpha}$) is the corresponding $p$-derivative in question defined in \eqref{eq005} (resp. \eqref{eq01}). 
\end{cor}

We now proceed to formulating fundamental theorems of a calculus for generalized derivatives. We start with a weak maximum principle.

\begin{theorem}\label{th3.3}
Let $f : [a,b] \to \mathbb{R}$ be such that $D_{p}f(t)$ exists for every $t \in (a, b)$. If $f$ has a local maximum (resp. minimum) at $c \in (a,b)$ and for each $t \in (a, b)$,
 and for all sufficiently small $|h|$ we have, 
 \begin{equation}\label{eq3.7}
p(t,h)<t, \quad h<0 \quad(resp. \ p(t,h)>t, \quad h>0),
\end{equation}
\begin{equation}\label{eq3.8}
p(t,h)>t, \quad h>0 \quad(resp.\  p(t,h)<t, \quad h<0),
\end{equation}
 then $D_{p}f(c)=0$.
\end{theorem}
\begin{proof}
We give the proof in the case where $f(c)$ is a local maximum. By hypothesis \eqref{eq3.8}, there exists $\delta_{0}$ such that 
$$p(t,h)>t, \quad 0< h <\delta_{0}.$$
Since $$p(c,h)>c, \quad 0< h <\delta_{0},$$ and $f(c)$ is a maximum, we can conclude that $$f(p(c,h))\le f(c).$$ Therefore $$\lim_{h \to 0^{+}}\frac{f(p(c,h))-f(c)}{h}\le 0.$$ Hence, $D_{p}f(c)\le 0$.
Next, by hypothesis \eqref{eq3.7}, there exists $\delta_{1}$ such that 
$$p(c,h)<c, \quad -\delta_1 < h < 0.$$ As before, since $f(c)$ is a maximum,  $$\lim_{h \to 0^{-}}\frac{f(p(c,h))-f(c)}{h}\ge 0.$$ So, $D_{p}f(c)\ge 0$. However, since $D_{p}f(c)$ exists, it follows that $D_{p}f(c)= 0$. The proof in the case where $f(c)$ is a minimum is similar, and therefore omitted.
\end{proof}

Next, we present a weak form of a general mean value theorem.
\begin{theorem}\label{weakmvt} {\rm [A generalized Mean Value Theorem]}
Let $p$ satisfy the conditions of Theorem~\ref{th3.3} and let $f$ be continuous on $[a, b]$. As an additional condition on $p$, assume that the $p$-derivative of the function whose values are $t-a$ exists for all $t\in (a, b)$. Then there is a constant $c \in (a, b)$ such that 
\begin{equation}\label{mvtc}
D_p f(c) = \frac{f(b)-f(a)}{b-a} k,
\end{equation}
where the constant $k$ is the $p$-derivative of the function whose values are $t-a$ evaluated at $t=c$, i.e., $ k = D_p(t-a) \bigg |_{t=c}.$
\end{theorem}
\begin{proof} Define $$h(t) = f(t) - f(a) - \frac{f(b)-f(a)}{b-a} (t-a).$$ Then, by hypothesis, $h$ is $p$-differentiable on $(a, b)$, continuous on $[a,b]$ and 
$$D_p h (t) = D_p f(t) - \frac{f(b)-f(a)}{b-a} D_p (t-a).$$
In the event that $D_p h(t) = 0$ for all $t \in (a, b)$, $c$ can be chosen to be any point in $(a, b)$. If $D_p h(t) \neq 0$ for all such $t$ we may assume, without loss of generality, that $h(t) > 0$ somewhere. Since $h(a)=h(b)=0$, and $h$ is continuous, it must achieve its maximum at, say, $t=c \in (a, b)$. By Theorem~\ref{th3.3}, $D_p h(c)=0$. The result follows.
\end{proof}
\begin{theorem} \label{th3.5y} If, in addition to the hypotheses in Theorem~\ref{weakmvt}, we have $p(t, 0)=t$, for every $t$, and $p_h(t, 0)$ exists, then 
$$ D_p f(c) = \frac{f(b)-f(a)}{b-a}\, p_h(c, 0).$$
\end{theorem}
\begin{proof}
Since $$D_p(t-a) \bigg |_{t=c} = \lim_{h\to 0} \frac{p(t,h)-a -(t-a)}{h} = \lim_{h\to 0} \frac{p(t,h) -p(t,0)}{h} = p_h(t,0),$$ the result follows.
\end{proof}

In the next two sections we can formulate more precise results by considering the two cases, $p_h(t, 0)\neq 0$ and $p_h (t, 0)=0.$ 

\section{The case where $p_h(t,0)\neq 0$.}

We'll see below that this restriction on $p$ is the most common one and all fractional derivatives presented (i.e., \eqref{eq001}, \eqref{eq002}, \eqref{eq003}, \eqref{eq002x}, \eqref{eq002y}, and \eqref{eq009}) satisfy this condition. 

The next result, when combined with Theorem~\ref{th1}, allows us to transform $p$-derivatives into ordinary derivatives.
\begin{prop}\label{th2.5} {\rm (See Theorem 2.4 in \cite{JPA})}
Let $p$ satisfy hypothesis (H). In addition, for $t\in I$, let 
\begin{equation}\label{eq6x}
\lim_{\varepsilon\to 0} \frac{\varepsilon}{h(t,\varepsilon)} \neq 0.
\end{equation}  Then $f$ is differentiable at $t$ iff and only if  $f$ is $p$-differentiable at $t$. In addition, 
\begin{equation}\label{eq5x}
D f(t) = p_h(t,0)\, f^{\prime}(t).
\end{equation}
\end{prop}
\begin{remark}
Violation of either \eqref{eq6x} or the tacit assumption, $p_h(t,0)\neq 0$, can void \eqref{eq5x}, see Remark 2.4 in \cite{JPA} and the example therein.
\end{remark}

\begin{cor}\label{th2.6} {\rm (See Theorem 2.2 in \cite{kha}; Theorem 2.3 in \cite{kat}; and Theorem 1 in \cite{as}.)}
Let the $\alpha$-derivative be defined as in either \eqref{eq001} or \eqref{eq002} and let $f$ be $\alpha$-differentiable at $t$. Then,
$$T_0^\alpha f(t) = t^{1-\alpha}\, f^\prime (t).$$ 
If $f$ is $\alpha$-differentiable in the sense of \eqref{eq003} then,
\begin{equation}\label{eq5xx}
D^{GFD} f(t) = \frac{\Gamma(\beta)}{\Gamma(\beta-\alpha+1)}\, t^{1-\alpha}\, f^\prime (t).
\end{equation}
If $f$ is $\alpha$-differentiable in the sense of \eqref{eq009} then, for $t\in [0,b]$, $b < \pi/2$,
\begin{equation}\label{eq5xy}
Df(t) = (\cos t)^{1-\alpha}\, f^\prime (t).
\end{equation}
Similar results hold for derivatives defined by either \eqref{eq002x} or \eqref{eq002y} (if $F(t,\alpha)\neq 0.$)
\end{cor}

\begin{cor}\label{th2.7}
Let $t>0$ (resp. $t\geq 0$). Then $f$ is $\alpha$-differentiable at $t$ in the sense of anyone of \eqref{eq001}, \eqref{eq002}, or \eqref{eq003}, (resp. \eqref{eq009}) if and only if $f$ is differentiable at $t$.
\end{cor}

Stronger versions of a generalized mean value theorem follow in which we do not require the assumptions in Theorem~\ref{th3.3} above but do require that $p_h(t,0)\neq 0$.
\begin{theorem}\label{cor3.6}{\rm [A generalized Mean Value Theorem]}
Let $p$ satisfy the conditions of Proposition~\ref{th2.5} and let $f$ be $p$-differentiable on $(a, b)$ and continuous on $[a, b]$.  Then there exists $c \in (a,b)$ such that $$D_{p}f(c)=\left[\frac{f(b)-f(a)}{b-a}\right]p_{h}(c,0).$$ 
\end{theorem}
\begin{proof} The proof is clear on account of the usual Mean Value Theorem applied to $f$ on $(a, b)$ since $f$ is necessarily differentiable there by Proposition~\ref{th2.5}. Since there exists $c\in (a, b)$ such that $f(b)-f(a) = (b-a)\,f^\prime(c)$ we get $D_{p}f(c)=p_h(c,0)\, f^\prime(c)$
and the result follows.
\end{proof}
\begin{theorem}\label{th2.6y} {\rm [Another generalized Mean Value Theorem]}
Let $p$ satisfy the conditions of Proposition~\ref{th2.5} and let $f, g$ be $p$-differentiable on $(a, b)$, continuous on $[a, b]$, and $D_p(g(t)) \neq 0$ there. Then there exists $c \in (a, b)$ such that
$$\frac{f(b)-f(a)}{g(b)-g(a)} = \frac{D_p f(c)}{D_p g(c)}.$$ 
\end{theorem}
\begin{proof} Write $$h(t) = f(t)-f(a)-\left[\frac{f(b)-f(a)}{g(b)-g(a)}\right](g(t)-g(a)).$$ then $h(a)=h(b)=0$ and $h$ satisfies the conditions of Theorem~\ref{cor3.6}. So, there exists $c \in (a, b)$ such that $D_p h(c) = 0$. However,
$$D_p h(c) = D_p f(c) - \left[\frac{f(b)-f(a)}{g(b)-g(a)}\right] D_p g(c).$$ The result follows.
\end{proof}
\begin{remark}
Specializing to the case where $$p(t, h) = t+\frac{\Gamma(\beta)\,  h t^{1-\alpha}}{\Gamma(\beta-\alpha+1)}$$ and $g(t) = t^\alpha/\Gamma(\alpha)$ with $\alpha \in (0, 1)$ we get [\cite{as}, Theorem 6]. The choice $g(t) =t^\alpha/\alpha$ gives [\cite{kat}, Theorem 2.9].

Next, if $p_h(t,0)$ exists everywhere on $(a, b)$ and $p(t,0)=0$, then Theorem~\ref{th3.5y} gives us that $k$ in \eqref{mvtc} is given by $k = p_h(c,0)$. In this case, we note that the function $f$ need not to be differentiable in the usual sense here (see Theorem~\ref{th2.5}) and $p_{h}(c,0)$ may or may not be zero.
\begin{exam}
Given $I=[a,b]$, $f(t)=|t|$ and $p(t,h)=t+th+t^{3}h^{3}$. Then $p$ satisfies the conditions of Theorem~\ref{weakmvt}. Furthermore, $p_{h}(t,0) \neq 0$  for $t \neq 0$. A simple calculation shows that $D_{p}f(t)=|t|=f(t)$, for all $t \in (a,b)$. By Theorem~\ref{weakmvt}, there exists $c \in (a,b)$ that $$D_{p}f(c)=\left[\frac{f(b)-f(a)}{b-a}\right]p_{h}(c,0),$$ i.e., $$|c|=\left[\frac{f(b)-f(a)}{b-a}\right]c.$$ 
 The existence of $c$ can be calculated directly as follows.  Let $a<b<0$. Then,  $$|c|=\left[\frac {-b+a}{b-a}\right]c,$$ so that $|c|=-c$. So, we may choose any $c$ such that $ a < c < b$. Let $0<a<b$. In this case, $$|c|=\left[\frac{b-a}{b-a}\right]c.$$ It suffices to choose $c$ such that $a < c <b$ again.  Finally, let $a<0<b$. As $$|c|=\left[\frac{b+a}{b-a}\right]c,$$ it suffices to choose $c=0$.
\end{exam}
\begin{remark}
In this example, $p_{h}(0,0)=0$ and consequently $D_{p}f(0)=0$ even though $f^{\prime}(0)$ does not exist.
\end{remark}

Of course, Rolle's theorem is obtained by setting $f(a)=f(b)=0$ in Theorem~\ref{th2.6y}. The latter then includes [\cite{kat}, Theorem 2.8].
\end{remark}

\begin{definition}\label{defx1}
Let $p$ satisfy (H), and let $f : [a,b] \to \mathbb{R}$ be continuous. Then $$I_{p}(f)(t)=\int_a^{t}\frac{f(x)}{p_{h}(x,0)}\, dx.$$
\end{definition}

This definition includes the fractional integral considered in [\cite{kat}, Definition 3.1]. Observe that, since $1/p_h \in L(a, b)$ and $f$ is bounded, this integral always exists (absolutely). It follows that  $I_{p}(f) \in AC[a,b]$ and consequently $I^{\prime}_{p}(f)$ exists a.e. In this case, the continuity of $p_h$ guarantees that $I_{p}(f) \in C^1(a, b)$.

Next we state and prove a version of the generalized fundamental theorem of calculus for such $p$-derivatives.
The first part is clear, i.e., 
\begin{theorem}\label{th3.5}
Let $p$ satisfy (H), and let  $f : [a,b] \to \mathbb{R}$ be continuous. Then $D_{p}(I_{p}(f)(t))=f(t).$
\end{theorem}
\begin{proof} By Proposition~\ref{th2.5}, since $I_p(f)$ is differentiable, we have,
\begin{eqnarray*}
D_{p}(I_{p}(f)(t)) &=& p_{h}(t,0)I^{\prime}_{p}(f)(t) =  f(t).
\end{eqnarray*}
\end{proof}
{\bf NOTE:} The preceding includes [\cite{kat}. Theorem 3.2] as a special case.
\begin{theorem}\label{th3.8}
Let $p$ satisfy (H) and let $F : [a,b] \to \mathbb{R}$ be continuous. If $F$ is $p$-differentiable on $(a, b)$ and $D_{p}F$ is continuous on $[a, b]$, then $I_{p}(D_{p}F)(b)=F(b)-F(a).$
\end{theorem}
\begin{proof}
Let $a=x_{0}<x_{1}<x_{2}<....<x_{n}=b$ be a partition of $[a,b]$. Applying corollary~\ref{cor3.6} to each $[x_{i-1}, x_{i}]$ we get, for some $t_i$, $$D_{p}F(t_{i})=\frac{F(x_{i})-F(x_{i-1})}{x_{i}-x_{i-1}}p_{h}(t_{i},0)$$ or $$F(x_{i})-F(x_{i-1})=(x_{i}-x_{i-1})\frac{D_{p}F(t_{i})}{p_{h}(t_{i},0)}.$$ Thus,  
\begin{eqnarray*}
F(b)-F(a) &=& \sum _{i=1}^n F(x_{i})-F(x_{i-1})= \sum _{i=1}^n \frac{D_{p}F(t_{i})}{p_{h}(t_{i},0)} \Delta x_{i}.
\end{eqnarray*}
Now since $f$ is continuous on every subinterval $[x_{i-1},x_{i}]$ of $[a, b]$, we can pass to the limit as $\Delta x_i \to 0$. This gives, $$F(b)-F(a)=\int_a^{b}\frac{D_{p}F(t)}{p_{h}(t,0)}\, dt.$$
 This shows that $I_{p}(D_{p}F)(b)=F(b)-F(a)$ and we are done.
\end{proof}
Combining Theorem~\ref{th3.5} and Theorem~\ref{th3.8} we get the generalized Fundamental Theorem of Calculus. In addition, using the above relation, we can get a Generalized Integration by Parts formula, i.e.,
\begin{cor} If $f, g$ are both $p$-differentiable on $(a, b)$ and continuous on $[a, b]$, then,
$$I_{p}(f D_{p}(g))=fg-I_{p}(D_{p}(f)g).$$ 
\end{cor}
This is clear on account of the Product Rule in Proposition~\ref{th2.2}(b) above. 

\section{The case where $p_h(t,0)=0$.}

In this section we consider the {\it exceptional case}
\begin{equation}\label{eqq04}
 p_h(t,0) =0.
\end{equation}
The effect of \eqref{eqq04} is that it tends to {\it smooth out} discontinuities in the ordinary derivative of functions. A glance at \eqref{eqq03} would lead one to guess that whenever \eqref{eqq04} holds we have $D_pf(t) =0$ but that is not the case, in general. 
\begin{exam}\label{ex5.1x} Consider the special case $p(t,h) = t+h^2$ which satisfies \eqref{eqq04}. Then the function $f(x) = \sqrt{x}$, $x>0$, although not differentiable at $x=0$, is clearly right-$p$-differentiable at $x=0$ with  $D_p^+f(0)=1$. 
\end{exam}
\begin{theorem}\label{th5.1xx}
Let \eqref{eqq04} hold, $p(t, 0)=t$, and assume that \eqref{eq3.7}, \eqref{eq3.8} are satisfied for each $t$, as well. If $f$ is continuous on $[a, b]$ and $D_p f(t)$ exists, then $D_p f(t)=0$.
\end{theorem}
\begin{proof}
Note that $f$ is continuous on $(a, b)$ on account of the hypothesis and Theorem~\ref{th1x}. Using the proofs of Theorem~\ref{weakmvt} and Theorem~\ref{th3.5y} we observe that the function $h$ defined there is continuous on $[a, b]$, as $f$ is continuous there, and therefore its maximum value is attained at $x=c$. Thus $k=p_h(c,0)=0$ by \eqref{mvtc}.
\end{proof}
Of course, the previous example had an ordinary derivative with an infinite discontinuity at $x=0$ but still simple discontinuities in the ordinary derivative can lead to the existence of their $p$-derivative for certain $p$. 
\begin{remark}
Incidentally, Example~\ref{ex5.1x}  also shows that \eqref{eq3.7} cannot be waived in the statement of Theorem~\ref{th5.1xx}.
\end{remark}
\begin{exam}
As before we let $p(t,h) = t+h^2$. Then $f(x) = |x|$, and although it is not differentiable at $x=0$,  it is clearly $p$-differentiable at $x=0$ with $D_p f(0)=0$.
\end{exam}

Below we study the consequences of this extraordinary assumption \eqref{eqq04} and its impact on the study of such $p$-derivatives.

\subsection{Consequences of \eqref{eqq04}}

We have seen that the notion of $p$-differentiability can be used to turn non-differentiable functions into $p$-differentiable ones, for some exceptional $p$'s and these can have a p-derivative equal to zero, as well. We first look at some simple special cases of $p$ satisfying \eqref{eqq04}. 

As is usual we define a {\it polygonal function} as a function whose planar graph is composed of line segments only, i.e., it is piecewise linear.
\begin{theorem}\label{th2.22}
Let $p(t,h) = t+h^2$. Then every polygonal function $f$ on $\R$ is $p$-differentiable everywhere and $D_p f(x)=0$ for all $x \in \R$.
\end{theorem}

\begin{proof}
Since the graph of every polygonal function consists of an at most countable and discrete set of simple discontinuities in its ordinary derivative, it is easy to show that its $p$-derivative at the cusp points must be zero (just like the absolute value function above). The curve being linear elsewhere it is easy to see that at all such points its $p$-derivative exists and must be zero (see Example 3.1 in \cite{JPA}).
\end{proof}

\begin{remark}
In contrast with the case where $p_h(t,0)\neq0$ where an integral can be defined via Definition~\ref{defx1}, in this case such an inverse cannot be constructed, in general, as the preceding example shows. 
\end{remark}

There must be limitations to this theory of generalized or $p$-derivatives. Thus, we investigate the non-existence of $p$-derivatives under condition \eqref{eqq04} for a class of power functions defining the derivative. Our main result is Theorem~\ref{th2.2} below which states that for power-like $p$-functions there are functions that are nowhere $p$-differentiable on the real line. In the event that $p_h(t,0)\neq 0$, Proposition~\ref{th2.5} makes it easy to construct functions that are nowhere $p$-differentiable on the whole real axis simply by choosing, in particular, any function with $p_h(t,0)=1$. For a fascinating historical survey of classical nowhere differentiable functions, the reader is encouraged to look at \cite{jt} (available online).

At this point one may think that $p$-differentiability is normal and that most functions have a zero $p$-derivative if $p_h(t, 0)=0$. This motivates the next question. {\it Does there exist a function $p$ satisfying \eqref{eqq04} such that it is continuous and nowhere $p$-differentiable on $\R$?} The answer is yes and is in the next theorem. 

\subsection{Weierstrass' continuous, nowhere differentiable function}

In this subsection we show that the series \eqref{kw},  first considered by Weierstrass, and one that led to a continuous nowhere differentiable function, \cite{wr}, can also serve as the basis for a continuous nowhere $p$-differentiable function for a large class of functions $p$ satisfying \eqref{eqq04}, namely power functions. Below we show that for each $\alpha > 1$ there is a function $p$ satisfying \eqref{eqq04} and a function $f$ that is nowhere $p$-differentiable on the whole line.

\begin{theorem}
Let $p(t,h) = t+h^\alpha$ where $\alpha > 1$. Then Weierstrass' continuous and nowhere differentiable function 
\begin{equation}\label{kw}
f(x) = \sum_{n=0}^\infty \, b^n\,\cos (a^n\pi x)
\end{equation}
 where $0 < b < 1$, $a$ is a positive integer, and 
\begin{equation}\label{aaa}
\sqrt[\alpha]{a}\,b > 1 + \frac32\,\pi,
\end{equation}
is nowhere $p$-differentiable on $\R$.
\end{theorem}
\begin{proof}
Observe that the cases where $ \alpha \leq1$ are excluded by \eqref{eqq04}, so we let $\alpha > 1$. We will show, as usual,  that there exists a sequence of $h \to 0$ along which $|(f(x+h^\alpha)-f(x))/h| \to \infty$. Now, for fixed $x \in \R$, 
\begin{eqnarray*}
\frac{f(x+h^\alpha) - f(x)}{h} &=& \sum_{n=0}^\infty b^n\, \frac{\cos (a^n\pi (x+h^\alpha)) - \cos (a^n\pi x)}{h}\\
&=&\sum_{n=0}^{m-1} b^n\, \frac{\cos (a^n\pi (x+h^\alpha)) - \cos (a^n\pi x)}{h} + \sum_{n=m}^\infty b^n\, \frac{\cos (a^n\pi (x+h^\alpha)) - \cos (a^n\pi x)}{h}\\
&:=& S_m + R_m.
\end{eqnarray*}
Estimating $S_m$ by the Mean Value Theorem shows that for some $0< \theta< 1$,
\begin{equation}\label{eqq05}
| \cos (a^n\pi (x+h^\alpha)) - \cos (a^n\pi x) | = |h^\alpha a^n \pi \sin (a^n\pi (x+\theta h^\alpha))| \leq a^n\pi |h|^\alpha,
\end{equation}
so that \begin{equation}\label{eqq06}|S_m | \leq \pi |h|^{\alpha -1}\,\sum_{n=0}^{m-1} (ab)^n  < \frac{\pi\, |h|^{\alpha -1} (ab)^m}{ab-1}.\end{equation}
Recall that $x$ is fixed at the outset. Now, for any positive integer $m$, we can write $a^mx $ in the form $a^m x = \alpha_m + t_m$ where $\alpha_m$ is an integer and $|t_m| \leq 1/2$. Define a sequence $h_m > 0$ by 
$$h_m = \sqrt[\alpha]{\frac{1-t_m}{a^m}}.$$
Then $0 < h_m^\alpha \leq 3/(2a^m).$ From this choice of a sequence and \eqref{eqq06} we get the estimate,

\begin{equation}\label{eqq07} 
|S_m| <  \pi\, \left (\frac{1-t_m}{a^m} \right )^{\frac{\alpha-1}{\alpha}}\, \frac{(ab)^m} {ab-1}
\leq \pi\, \left (\frac{3}{2a^m} \right )^{\frac{\alpha-1}{\alpha}}\, \frac{(ab)^m} {ab-1}
=  \pi\, \left (\frac{3}{2} \right )^{\frac{\alpha-1}{\alpha}} \frac{ {a^{\frac{m}{\alpha}}}b^m }{ab-1}.
\end{equation}

The next step is to show that the remainder term, $R_m$, remains bounded away from $0$. To this end note that $a^n\, \pi\,(x+h_m^\alpha) = a^{n-m} a^m\pi\, (x+h_m^\alpha) = a^{n-m}\pi (\alpha_m + 1).$ It follows that since $a$ is odd, then for $n \geq m$, we have 
\begin{equation}\label{eqq08}
\cos (a^n\, \pi\,(x+h_m^\alpha)) = (-1)^{\alpha_m+1}.
\end{equation}
A similar calculation shows that 
\begin{equation}\label{eqq09a}
\cos (a^n\, \pi\, x) = \cos (a^{n-m}\, \pi\, (\alpha_m + t_m)) = (-1)^{\alpha_m} \cos(a^{n-m}\pi t_m).
\end{equation}
Combining \eqref{eqq08} and \eqref{eqq09a} we see that
\begin{eqnarray*}
 R_m &=& \sum_{n=m}^\infty b^n \frac{(-1)^{\alpha_m +1} - (-1)^{\alpha_m}\cos (a^{n-m}\pi t_m)}{h_m}\\
&=& \frac{(-1)^{\alpha_m +1}}{h_m}\sum_{n=m}^\infty b^n (1 + \cos (a^{n-m}\pi t_m))\\
{\rm i.e., }\ |R_m| &= & \frac{1}{|h_m|}\sum_{n=m}^\infty b^n (1 + \cos (a^{n-m}\pi t_m)).\\
\end{eqnarray*}
Since the previous series is a series of non-negative terms we can drop all terms except the first. In this case note that $\cos(\pi t_m) \geq 0$ since $|t_m|\leq 1/2$. So, 
\begin{equation}\label{eqq09}
|R_m| > \frac{b^m}{|h_m|} > \sqrt[\alpha]{\frac23} a^{m/\alpha}\, b^m.
\end{equation}
Finally, using \eqref{eqq09} and \eqref{eqq07} we get
\begin{gather}
\bigg | \frac{f(x+h_m^\alpha) - f(x)}{h_m}\bigg | \geq  |R_m | - |S_m |
 > \sqrt[\alpha]{\frac23} a^{\frac{m}{\alpha}}\, b^m - \pi\, \left (\frac{3}{2} \right )^{\frac{\alpha-1}{\alpha}} \frac{ {a^{\frac{m}{\alpha}}}b^m }{ab-1}
 = \left (\sqrt[\alpha]{\frac23}  - \frac{\pi}{ab-1}\,\left (\frac{3}{2} \right)^{\frac{\alpha-1}{\alpha}} \right) a^{\frac{m}{\alpha}} b^m  \label{eqq10}
\end{gather}
Since $b<1$ we must have $a \geq 3$ so that $ab > \sqrt[\alpha]{a}b$. The stronger hypothesis \eqref{aaa} forces both $\sqrt[\alpha]{a}b>1$ and the term in the parentheses in \eqref{eqq10} to be positive. Since $h_m \to 0$ as $m \to \infty$, the left hand side of \eqref{eqq10} tends to infinity, so that the resulting $p$-derivative cannot exist at $x$. Since $x$ is arbitrary, the conclusion follows.
\end{proof}

\section{An existence and uniqueness theorem}

In the final section we give conditions under which an initial value problem for a generalized Riccati equation with $p$-derivatives has a solution that exists and is unique. 

\begin{theorem}
Let $p$ satisfy (H) and $q : [0,T] \to \mathbb{R}, T<\infty$ be continuous. Assume that for some  $b>0,$ we have 
\begin{equation}\label{hypx}
\|\frac{1}{p_{h}}\|_{L^{1}[0,T]}< \min \left\{ \frac{b}{\|q\|_\infty+b^{2}}, \frac{1}{2b} \right\}.
\end{equation}
Then the initial value problem for the (generalized) Riccati differential equation 
\begin{equation}\label{rde}
D_p u(t)+u^{2}(t)=q(t),\quad u(0)=u_{0},
\end{equation}
 has a unique continuous solution $u(t)$ on $[0, T]$.
\begin{proof} Let $B=\{u \in C[0,T], \|u\|_\infty\leq b\}$. Then $B$ is a complete metric space.
Define an operator $F$ on $B$ by $F(u)=I_{p}(q(t)-u^{2}(t))+u_{0}.$
Then for every $u, v \in B$.
\begin{eqnarray*}
|Fu-Fv| &=& | I_{p}(q(t)-u^{2}(t))-I_{p}(q(t)-v^{2}(t))| \\ 
&=& \bigg |\int_0^t \frac{q(s)-u^{2}(s)}{p_{h}(s,0)}-\frac{q(s)-v^{2}(s)}{p_{h}(s,0)}\, ds\bigg | \\ 
&=& \bigg |\int_0^t \frac{(v(s)-u(s))(v(s)+u(s))}{p_{h}(s,0)}\, ds \bigg | \\ 
&\le& 2b \|u-v\|_\infty \, \int_0^t \bigg |\frac{1}{p_{h}(s,0)} \bigg| \, ds 
\end{eqnarray*}
It follows that $\|Fu-Fv\|_\infty<k\|u-v\|$ is a contraction on $B$, with $k= 2b \|\frac{1}{p_{h}}\|_{L^{1}[0,T]}<1$, by hypothesis. 

Next, we show that $F:B\to B$. Clearly, for $u\in B$, $Fu$ is continuous on $[0,T]$. Next, observe that
\begin{eqnarray*}
\|Fu\|_\infty \leq \int_0^T \frac{\| q(s) - u^2(s)\|_\infty}{|p_h(s,0)|}\, ds \leq (\|q\|_\infty + b^2) \|\frac{1}{p_{h}}\|_{L^{1}[0,T]} \leq b,
\end{eqnarray*}
by hypothesis. Hence $F$ maps $B$ into itself. Applying the contraction principle we get that $F$ has a unique fixed point $u \in C[0,T]$ such that $Fu=u.$ Theorem~\ref{th3.5} gives us the final result. 
\end{proof}
\end{theorem}
\begin{remark}
Observe that there are no sign restrictions on $p_h(t, 0)$. Note that $D_p$ may in fact depend on a parameter $\alpha$, subject only to the $L^1$-condition on $1/p_h$ at the outset. For example, if we choose $p(t,h) = t+ht^{1-\alpha}$ as in \cite{kha}, the hypothesis \eqref{hypx} above becomes,
$$\frac{T^\alpha}{\alpha} \leq \min \left\{ \frac{b}{\|q\|_\infty+b^{2}}, \frac{1}{2b} \right\},$$
so we can see that the assumption that $\alpha \in (0,1)$ is not necessary, just that $\alpha >0$. Of course, $T$ will generally
decrease as $\alpha$ grows. Finally, this solution can always be found using the method of successive approximations as implied by the contraction principle.

Similarly, if $p(t,h)=t + \frac{\Gamma(\beta)}{\Gamma(\beta-\alpha+1)} h t^{1-\alpha}$ with $\beta>-1$, $\beta \in \mathbb{R^{+}}$ and $0<\alpha \le 1$ as in \cite{as}, the generalized Rolle's Theorem, Mean Value Theorem and Riccati differential equation studied here include the corresponding theorems in \cite{as}. In addition, this existence theorem clarifies the purely numerical results obtained in \cite{as} when solving a special Riccati equation of the form \eqref{rde} using the fractional derivative \eqref{eq003}, which, as we have shown, is contained in our theory.
\end{remark}

\section{Open question}
\vskip0.15in
\begin{enumerate}
\item 
Is there a function $p$, satisfying \eqref{eqq04}, and a function $f$ such that $f$ is $p$-differentiable and such that $D_p f(t) \neq 0$ for all t in some interval (or, more generally, some set of positive measure)?
\end{enumerate}

\section{Conclusion}
In this paper we have extended the theory of $p$-derivatives in \cite{JPA} to include results such as the mean-value theorem, Rolle's theorem and integration by parts. In so doing we pointed out that the conformable fractional derivative of a given function, as considered by \cite{kha} and \cite{kat}, is actually an ordinary derivative except possibly at one point. We expanded on the cases where the partial derivative $p_h(t,0)$ either vanishes or doesn't and in so doing showed that in the former case there exists, for each $\alpha >1$, a fractional derivative and a function whose fractional derivative exists nowhere on the real line. In the case where $p_h(t,0)=0$ many of the previous results have no analogues and an inverse of the $p$-derivative generally does not exist. We also presented an existence and uniqueness theorem for a Riccati-type equation involving a $p$-derivative whose solution may always be found using successive approximations. The results presented here extend many of the results found in the literature as referred to in the text.

\end{document}